\documentclass[11pt]{article}
\usepackage{amsthm, amsmath, amssymb, amsfonts, url, booktabs, tikz, setspace, fancyhdr, enumerate}
\usepackage[margin = 1in]{geometry}

\usepackage{hyperref}

\usepackage[textsize=scriptsize,colorinlistoftodos]{todonotes}

\newtheorem{theorem}{Theorem}[section]

\newtheorem{lemma}[theorem]{Lemma}

\newtheorem{conjecture}[theorem]{Conjecture}

\theoremstyle{definition}

\theoremstyle{remark}

\DeclareMathOperator{\ex}{ex}

\newcommand{\Hom}{\mathrm{Hom}}
\newcommand{\EE}{\mathbb{E}}
\newcommand{\PP}{\mathbb{P}}

\newcommand{\ext}{\mathrm{ex}}

\newcommand{\J}{\mathcal{J}}
\newcommand{\K}{\mathcal{K}}

\tikzstyle{p}+=[fill=black, circle, minimum width = 1pt, inner sep =
1pt]
\tikzstyle{w}+=[fill=white, draw, circle, minimum width = 1pt, inner sep =
1.5pt]

\begin{document}

\title{On the extremal number of subdivisions}

\author{David Conlon\thanks{Mathematical Institute, Oxford OX2 6GG, United
Kingdom.
E-mail: {\tt
david.conlon@maths.ox.ac.uk}. Research supported by a Royal Society University Research Fellowship and by ERC Starting Grant 676632.} \and Joonkyung Lee\thanks{Mathematical Institute, Oxford OX2 6GG, United
Kingdom.
E-mail: {\tt
joonkyung.lee@maths.ox.ac.uk}. Research supported by ERC Starting Grant 676632.}}

\date{}

\maketitle

\begin{abstract}
One of the cornerstones of extremal graph theory is a result of F\"uredi, later reproved and given due prominence by Alon, Krivelevich and Sudakov, saying that if $H$ is a bipartite graph with maximum degree $r$ on one side, then there is a constant $C$ such that every graph with $n$ vertices and $C n^{2 - 1/r}$ edges contains a copy of $H$. This result is tight up to the constant when $H$ contains a copy of $K_{r,s}$ with $s$ sufficiently large in terms of $r$. We conjecture that this is essentially the only situation in which F\"uredi's result can be tight and prove this conjecture for $r = 2$. More precisely, we show that if $H$ is a $C_4$-free bipartite graph with maximum degree $2$ on one side, then there are positive constants $C$ and $\delta$ such that every graph with $n$ vertices and $C n^{3/2 - \delta}$ edges contains a copy of $H$. 
This answers a question of Erd\H{o}s from 1988.
The proof relies on a novel variant of the dependent random choice technique which may be of independent interest.
\end{abstract}

\section{Introduction}

Given a graph $H$ and a natural number $n$, the \emph{extremal number} $\ext(n, H)$ is the largest number of edges in an $H$-free graph with $n$ vertices. The classical Erd\H{o}s--Stone--Simonovits theorem~\cite{ESi66, ES46} states that
\[\ext(n, H) = \left(1 - \frac{1}{\chi(H)-1} + o(1)\right)\binom{n}{2},\]
where $\chi(H)$ is the chromatic number of $H$. At first glance, this gives an entirely satisfactory answer to the problem of estimating $\ext(n,H)$. However, for bipartite graphs, where $\chi(H) = 2$, it only gives the bound $\ext(n, H) = o(n^2)$. Attempts to find more accurate bounds for various bipartite $H$, particularly complete bipartite graphs and even cycles, occupy a central place in extremal combinatorics. We refer the interested reader to~\cite{FS13} for a thorough and detailed survey.

One of the few general results in this area, first proved by F\"uredi~\cite{Fu91} while establishing a conjecture of Erd\H{o}s and later reproved by Alon, Krivelevich and Sudakov~\cite{AKS03} using the celebrated dependent random choice technique (see~\cite{FS11} and its references), is the following.

\begin{theorem}[F\"uredi, Alon--Krivelevich--Sudakov] \label{thm:Furedi}
For any bipartite graph $H$ with maximum degree $r$ on one side, there exists a constant $C$ such that 
\[\ext(n,H) \leq C n^{2 - 1/r}.\]
\end{theorem}

A result of Alon, R\'onyai and Szab\'o~\cite{ARSz99}, building on work in~\cite{KRS96}, says that if $K_{r,s}$ is the complete bipartite graph with parts of order $r$ and $s$, then $\ext(n, K_{r,s}) = \Omega(n^{2 - 1/r})$ for all $s > (r-1)!$. Therefore, if $H$ contains a copy of $K_{r,s}$ with $s > (r-1)!$, Theorem~\ref{thm:Furedi} is tight up to the implied constant. Moreover, if $\ext(n, K_{r,r}) = \Omega(n^{2 - 1/r})$, as is believed (at least by some~\cite{KST54}), F\"uredi's result would be tight for any $H$ containing a copy of $K_{r,r}$. We conjecture that this is the only way in which it can be tight.

\begin{conjecture} \label{conj:main}
For any bipartite graph $H$ with maximum degree $r$ on one side containing no $K_{r,r}$, there exist positive constants $C$ and $\delta$ such that 
\[\ext(n,H) \leq C n^{2 - 1/r - \delta}.\]
\end{conjecture}

The principal motivation behind Conjecture~\ref{conj:main}, and the main result of this paper, is that it holds for $r = 2$.

\begin{theorem}\label{thm:1subdiv}
For any bipartite graph $H$ with maximum degree two on one side containing no $C_4$, there exist positive constants $C$ and $\delta$ such that
\[\ext(n,H) \leq C n^{3/2 - \delta}.\]
\end{theorem}

Another way of viewing Theorem~\ref{thm:1subdiv} is in terms of subdivisions. The \emph{$k$-subdivision} $H^{k}$ of a graph $H$ is the graph obtained from $H$ by replacing the edges of $H$ with internally disjoint paths of length $k+1$. When $k = 1$, we simply talk about the \emph{subdivision} of $H$. Theorem~\ref{thm:1subdiv} is then equivalent to the statement that for any graph $H$, there exist positive constants $C$ and $\delta$ such that the extremal number of the subdivision of $H$ is at most $C n^{3/2 - \delta}$. Moreover, since $H$ is contained in the clique on $v(H)$ vertices, it clearly suffices to establish the required estimate for the subdivision of cliques.

We are by no means the first people to study the extremal number of subdivisions. For instance, even cycles, which occupy a central place in extremal graph theory~\cite{V16}, are themselves subdivisions. A more explicit mention of subdivisions was made by Erd\H{o}s~\cite{Er88} in 1988, when he asked whether, for every $t \geq 3$, there exist $\alpha<3/2$ and $C>0$ such that the extremal number of the subdivision of $K_t$ is at most $Cn^{\alpha}$. 
Theorem~\ref{thm:1subdiv} answers this old question affirmatively.

For longer subdivisions, the extremal numbers were studied in detail by Jiang and Seiver~\cite{JS12}, building on earlier work on the extremal function for the closely-related notion of topological minor~\cite{J11, KP88}. They showed that if $k$ is odd, then $\ext(n, K_t^{k}) = O_{k,t} (n^{1 + 16/(k + 1)})$, an estimate which has the correct form, though the constant $16$ in the exponent is certainly not best possible (see the concluding remarks for some more discussion on this point). Further related results may be found dotted throughout the literature, for example in~\cite{A94Sub, AKS03, FS09}.

Like many of these advances, our approach to proving Theorem~\ref{thm:1subdiv} is based on applying the dependent random choice technique. However, the means by which we apply this technique is quite non-standard. Very roughly, we split into two distinct cases. The first case is when there is a large subgraph which contains significantly more $C_4$'s than one would expect in a random graph of the same density. In this case, we can apply a novel variant of dependent random choice, described in Section~\ref{sec:DRC}, to show that the graph contains the required graph $H$. On the other hand, if no large subgraph contains too many $C_4$'s, we know that the copies of $K_{1,2}$ must be quite well-distributed throughout, a property which will again be enough to find the required copy of $H$ (though several difficulties arise in making this intuition rigorous). We now begin the details in earnest.

\section{Preliminaries}

In this section, we collect a few results that will be useful to us in what follows. The first such result says that every sufficiently dense graph contains a large subgraph which is almost regular and such that the number of edges is a similar function of the number of vertices as in the original graph. This is essentially due to Erd\H{o}s and Simonovits~\cite[Theorem 1]{ES70}, though we use a slight variant noted by Jiang and Seiver~\cite[Proposition 2.7]{JS12}. In stating this result, we say that a graph $G$ is $K$-\emph{almost-regular} if $\max_{v\in V(G)}\deg(v)\leq K\min_{v\in v(G)}\deg(v)$.

\begin{lemma} \label{lem:ES}
For any positive constant $\alpha < 1$, there exists $n_0$ such that if $n \geq n_0$, $C \geq 1$ and $G$ is an $n$-vertex graph with at least $C n^{1 + \alpha}$ edges, then $G$ has a $K$-almost-regular subgraph $G'$ with $m$ vertices such that
$m\geq n^{\frac{\alpha(1-\alpha)}{2(1+\alpha)}}$, $|E(G')|\geq \frac{2 C}{5}m^{1+\alpha}$ and $K=20\cdot 2^{1+1/\alpha^2}$.
\end{lemma}

We will in fact need a bipartite version of this lemma, but this follows as a simple corollary of the original lemma and Chernoff's inequality, which we now recall.

\begin{lemma}\label{lem:chernoff}
Let $X$ be a binomial random variable with parameters $n$ and $p$, i.e., $X\sim \mathrm{Bin}(n,p)$, and let $\varepsilon\in (0,1)$.
Then
\begin{align*}
\PP [|X-\EE(X)|\geq \varepsilon \EE(X)]\leq 2\exp(-\varepsilon^2 \EE(X)/3).
\end{align*}
\end{lemma}

In the next result, our bipartite version of Lemma~\ref{lem:ES}, we say that a bipartite graph $G$ with bipartition $A\cup B$ is \emph{balanced} if $\frac{1}{2}|B|\leq |A|\leq 2|B|$. 

\begin{lemma}\label{lem:regular}
For any positive constant $\alpha < 1$, there exists $n_0$ such that if $n \geq n_0$, $C \geq 1$ and $G$ is an $n$-vertex graph with at least $C n^{1 + \alpha}$ edges, then $G$ has a $K$-almost-regular balanced bipartite subgraph $G''$ with $m$ vertices such that $m \geq n^{\frac{\alpha(1-\alpha)}{2(1+\alpha)}}$, $|E(G'')|\geq \frac{C}{10}m^{1+\alpha}$ and $K=60\cdot 2^{1+1/\alpha^2}$.
\end{lemma}

\begin{proof}
Let $L = 20\cdot 2^{1+1/\alpha^2}$. By Lemma~\ref{lem:ES}, there exists an $m$-vertex $L$-almost-regular subgraph $G'$ of $G$ such that $m\geq n^{\frac{\alpha(1-\alpha)}{2(1+\alpha)}}$ and $|E(G')|\geq \frac{2C}{5}m^{1+\alpha}$.
Let $A$ be a random subset of $V(G')$ taken by including each vertex $v\in V(G')$ in $A$ independently with probability $1/2$ and let $B=V(G')\setminus A$.
Then $\EE(|A|)=m/2$ and, thus, by Lemma~\ref{lem:chernoff} with $\varepsilon=1/3$,
\begin{align*}
\PP[m/3\leq |A|\leq 2m/3]\geq 1-2\exp(-m/54).
\end{align*}
Let $G''$ be the bipartite subgraph $G'[A,B]$. Since for each $v\in V(G')$, $\EE(\deg_{G''}(v))=\frac{1}{2}\deg_{G'}(v)$, Lemma~\ref{lem:chernoff} with $\varepsilon=1/2$ again implies that
\begin{align*}
\PP\left[\frac{1}{4}\deg_{G'}(v)\leq \deg_{G''}(v)\leq\frac{3}{4}\deg_{G'}(v)\right]
\geq 1-2\exp(-\deg_{G'}(v)/24)\geq 1-2\exp(-m^{\alpha}/72L).
\end{align*}
Hence, if $n$ is large enough to guarantee that $2\exp(-m/54)+2m\exp(-m^{\alpha}/72L)<1$, then with nonzero probability $G''$ is an $m$-vertex $3L$-almost-regular balanced bipartite graph with at least $|E(G')|/4 \geq \frac{C}{10} m^{1+\alpha}$ edges.
\end{proof}

Given a bipartite graph $G$ on $A\cup B$, the \emph{neighbourhood graph} on $A$
is the weighted graph $W$ on vertex set $A$, where the edge weight $W(u,v)$ is given by the codegree $d(u,v)$. We will need the following lemma, based on a similar result from~\cite{CKLL15}, saying that if every vertex in $A$ has high minimum degree in $G$, then the neighbourhood graph has substantial weight on every large subset of $A$. 

\begin{lemma}\label{lem:local_dense}
Let $G$ be a bipartite graph with bipartition $A\cup B$, $|B|=n$,
and minimum degree at least $\delta$ on the vertices in $A$. Then, for any subset $U\subseteq A$ with $\delta|U|\geq 2n$,
\begin{align*}
\sum_{uv\in\binom{U}{2}}d(u,v)\geq \frac{\delta^2}{2n}\binom{|U|}{2},
\end{align*}
where $d(u,v)$ is the codegree of $u$ and $v$.
\end{lemma}

\begin{proof}
Writing $d_U(v)$ for the degree of a vertex $v$ in $U$, we have that
\begin{align*}
\sum_{uv\in\binom{U}{2}}d(u,v) & = \sum_{b \in B} \binom{d_U(b)}{2} \geq n \binom{\sum_b d_U(b)/n}{2}\\
& = n \binom{\sum_u d(u)/n}{2} \geq n \binom{\delta |U|/n}{2} \geq \frac{\delta^2}{2n} \binom{|U|}{2},
\end{align*}
where the first inequality follows from the convexity of $\binom{x}{2}$ and the last inequality makes use of the assumption $\delta|U|\geq 2n$.
\end{proof}

We say that a weighted graph $W$ on vertex set $V$ is \emph{$(\rho,d)$-dense} if 
\begin{align*}
	\sum_{uv\in\binom{U}{2}}W(u,v)\geq d\binom{|U|}{2}
\end{align*}
for every $U\subseteq V$ with $|U|\geq\rho|V|$. Lemma \ref{lem:local_dense} then implies that the neighbourhood graph on $A$ is $(2n/\delta|A|,\delta^2/2n)$-dense, where $\delta$ is the minimum degree on $A$ and $n = |B|$.

We will also use the following elementary estimate.

\begin{lemma} \label{lem:ordered_pairs}
Let $G$ be a bipartite graph with bipartition $A\cup B$, $|B|=n$,
and minimum degree at least $\delta$ on the vertices in $A$. Then, for any subset $U\subseteq A$ with $\delta|U|\geq 2n$,
\begin{align*}
\sum_{uv\in\binom{U}{2}}d(u,v)\geq \frac{1}{4}\sum_{(u,v)\in U^2} d(u,v).
\end{align*}
\end{lemma}
\begin{proof}
Let $G':=G[U,B]$.
Since
\[
\sum_{uv\in\binom{U}{2}}d(u,v)=\frac{1}{2} \left(\sum_{(u,v)\in U^2} d(u,v)-e(G')\right),
\]
it is enough to prove that $e(G')\leq \frac{1}{2}\sum_{(u,v)\in U^2} d(u,v)$.
But, by convexity,
\[
\sum_{(u,v)\in U^2} d(u,v)=\sum_{b\in B}\deg_{G'}(b)^2\geq \frac{1}{n}e(G')^2.
\]
Thus, $e(G')\geq \delta |U|\geq 2n$ implies the desired bound.
\end{proof}

The following simple lemma, which we will apply repeatedly, gives a condition under which a weighted graph cannot have too many edges of high weight.

\begin{lemma}\label{lem:heavy}
Let $W$ be a weighted graph on vertex set $A$ and let
\begin{align*}
F_M :=\left\{uv\in \binom{A}{2}: W(u,v)\geq M\right\}.
\end{align*}
If $\sum_{uv\in \binom{A}{2}} W(u,v)^2 \leq S$, then
\begin{align*}
\sum_{uv\in F_M} W(u,v)\leq \frac{S}{M}.
\end{align*}
\end{lemma}

\begin{proof}
Since
\[M \sum_{uv\in F_M} W(u,v) \leq \sum_{uv\in \binom{A}{2}} W(u,v)^2 \leq S,\]
the required inequality holds.
\end{proof}

Finally, we will need a lemma saying that if much of the weight in a weighted graph is distributed along edges with low weight, then there is a large subset such that every vertex in this subset is connected to many low weight edges.

\begin{lemma}\label{lem:large_nbd}
Let $W$ be a weighted graph on vertex set $A$ and let
\begin{align*}
G_M :=\left\{uv\in \binom{A}{2}: W(u,v)< M\right\}.
\end{align*}
If $\sum_{uv\in G_M}W(u,v)\geq d |A|^2$, then there exists $U\subseteq A$ with $|U|\geq d|A|/M$ such that, for each $u\in U$,
\begin{align*}
|\{v\in A: 0<W(u,v)<M\}|\geq d|A|/M.
\end{align*}
\end{lemma}

\begin{proof}
Let $W'$ be the truncated weighted graph on $A$ with respect to the support $G_M$, 
i.e., $W'(u,v)=W(u,v)$ if $uv\in G_M$ and $W'(u,v)=0$ otherwise. Define
\begin{align*}
U:=\left\{u\in A: \sum_{v\in A}W'(u,v) \geq d|A|\right\}.
\end{align*}
Then
\begin{align*}
2 d|A|^2 \leq \sum_{u,v\in A} W'(u,v)
=\sum_{u\in U}\sum_{v\in A}W'(u,v)+\sum_{u\notin U}\sum_{v\in A}W'(u,v)
\leq M|U||A| + d|A|^2,
\end{align*}
which gives $|U|\geq d|A|/M$. For each $u\in U$,
\begin{align*}
d|A|\leq \sum_{v\in A}W'(u,v) \leq M|\{v\in A: 0<W(u,v)<M\}|,
\end{align*}
which implies the required estimate.
\end{proof}

\section{The key lemma} \label{sec:DRC}

Arguably the core of our proof is the following lemma, established by a variant of the dependent random choice technique, which says that if $G$ is a sufficiently large graph which contains significantly more $C_4$'s than one would typically expect in a random graph of the same density, then $G$ contains a copy of any fixed bipartite graph $H$ with maximum degree $2$ on one side. 

\begin{lemma}\label{lem:manyC4}
Let $H$ be a bipartite graph with maximum degree $2$ on one side and let $c$, $\varepsilon$, $K$ and $C$ be positive constants. Then there exists $n_0$ such that if $n \geq n_0$ and $G$ is an $n$-vertex $K$-almost-regular balanced bipartite graph on $A\cup B$ such that $e(G)=Cn^{3/2-c}$ and 
\begin{align*}
|\Hom(C_4,G)|\geq n^{2-2c+\varepsilon},
\end{align*}
then $G$ contains a copy of $H$.
\end{lemma}

To get some feeling for these numbers, note that if $|\Hom(C_4,G)|\geq n^{2-2c+\varepsilon}$, then the number of homomorphic copies of $C_4$ in $G$ vastly exceeds the number of copies of $K_{1,2}$, which is $O(n^{2-2c})$, so there must be many true copies of $C_4$. This points the way to the following generalisation, from which Lemma~\ref{lem:manyC4} clearly follows.

Recall that $\Hom(H, G)$ denotes the set of all homomorphisms from $H$ to $G$, that is, maps $\phi: V(H) \rightarrow V(G)$ such that $\phi(u) \phi(v) \in E(G)$ whenever $uv \in E(H)$. We will use $\Hom^*(K_{2,1},G)$ to denote the number of homomorphisms from $K_{2,1}$ to $G$ oriented so that the pair of vertices is always placed in $A$ and the single vertex in $B$. Similarly, $\Hom^*(K_{1,2}, G)$ will denote the number of homomorphisms oriented so that the pair of vertices is always placed in $B$ and the single vertex in $A$. We will also use the standard notation $N(x)$ for the set of neighbours of a vertex $x$ and $N(X)$ for the set of vertices adjacent to some vertex in a set $X$. Here and throughout, we will use $\log$ to denote the logarithm base $2$.

\begin{lemma}\label{lem:manyC4asym_2}
Let $H$ be a bipartite graph with maximum degree $2$ on one side. Then there exists $C$ such that if $G$ is a bipartite graph on $A\cup B$ with $n=|B| \geq |A|$, $M=C\log n$ and 
\begin{align*}
|\Hom(C_4,G)|\geq M |\Hom(K_{1,2},G)|,
\end{align*}
then $G$ contains a copy of $H$.
\end{lemma}

\begin{proof}
Note first that $|\Hom^*(K_{1,2},G)| =\sum_{a\in A} \deg(v)^2$,
\begin{align*}
|\Hom(C_4,G)|=\sum_{a\in A}\sum_{u,v\in N(a)} d(u,v)~~\text{ and }~~ 
|\Hom^*(K_{2,1},G)| =\sum_{a\in A} \sum_{v\in N(a)} \deg(v)
\end{align*}
and, thus,
\begin{align*}
\EE_{a\in A} \left[\sum_{u,v\in N(a)}d(u,v)
-M\left(\sum_{v\in N(a)}\deg(v) + \deg(a)^2\right)\right]\geq 0.
\end{align*}
Hence, there exists $x\in A$ such that
\begin{align}\label{eq:K12vsC4}
\sum_{u,v\in N(x)} d(u,v)\geq M\left(\sum_{v\in N(x)}\deg(v)
+\deg(x)^2\right).
\end{align}

Let $B':=N(x)$ and $A':=N(B')$. Define the dyadic partition $A_1\cup A_2\cup\cdots\cup A_L$ of $A'$ with $L = \lfloor \log |B'| \rfloor + 1 \leq 2\log n$ such that
\begin{align*}
A_i:=\{a\in A':2^{i-1}\leq\deg_{B'}(a)<2^i\}
\end{align*}
for each $i=1,2,\cdots,L$.
As $\sum_{a\in A'} \deg_{B'}(a)^2 = \sum_{u,v\in B'} d(u,v)$, by averaging there exists $j$ such that 
$\sum_{a\in A_j} \deg_{B'}(a)^2\geq \frac{1}{L}\sum_{u,v\in B'} d(u,v)$.
Since $\sum_{a\in A_j} \deg_{B'}(a)^2 < 2^{2j} |A_j|$, we have the bound 
\begin{align}\label{eq:dyadic_lower_asym}
|A_j|\geq \frac{1}{2^{2j}L}\sum_{u,v\in B'} d(u,v).
\end{align}
On the other hand, as $e(G[A'\cup B']) \leq \sum_{b\in B'}\deg(b)$, \eqref{eq:K12vsC4} implies that
\begin{align*}
\sum_{u,v\in B'} d(u,v)\geq M\sum_{b\in B'}\deg(b) \geq M|A_j|2^{j-1}.
\end{align*}
Combining this with \eqref{eq:dyadic_lower_asym}, it follows that
\begin{align}\label{eq:neighbour_large}
 2^j \geq M/2L.
\end{align}

We now apply dependent random choice to the graph $G'=G[A_j,B']$.
Let $h=|V(H)|$. 
We say that a pair of distinct vertices $u,v\in B'$ is \emph{bad} if $d(u,v)<h$.
Pick a random vertex $z\in A_j$ and let $X$ be the number of bad pairs in $N_{B'}(z)$, the neighbourhood of $z$ in $B'$. 
Then, using \eqref{eq:dyadic_lower_asym}, 
\begin{align*}
\EE[X]=\sum_{\text{bad }u,v\in B'} \PP[u,v\in N(z)]< \frac{h}{|A_j|}|B'|^2
\leq \frac{hL 2^{2j}|B'|^2}{\sum_{u,v\in B'}d(u,v)}.
\end{align*}
Since $M|B'|^2\leq \sum_{u,v\in B'}d(u,v)$ by \eqref{eq:K12vsC4}, we obtain
\begin{align*}
\EE[X]< \frac{hL2^{2j}}{M}.
\end{align*}
Therefore, the expected proportion of bad (ordered) pairs in $N_{B'}(z)$ is small, i.e., 
\begin{align*}
\EE\left[ \frac{X}{|N_{B'}(z)|(|N_{B'}(z)|-1)}\right] < \frac{hL2^{2j}}{M}\cdot\frac{1}{2^{2j-3}}
\leq 8 hL/M.
\end{align*}
For $C \geq 16 h^3$, there must then exist a choice of $z$ for which 
\[\frac{X}{|N_{B'}(z)|(|N_{B'}(z)|-1)} < 8 hL/M \leq 1/h^2.\]
Fixing such a $z$, we see, by \eqref{eq:neighbour_large}, that we can make $N_{B'}(z)$ as large as we wish by choosing $C$ sufficiently large. Moreover, the auxiliary graph on $N_{B'}(z)$ whose edge set consists of non-bad pairs has density larger than $1-1/(h-1)$. By Tur\'an's theorem, when $N_{B'}(z)$ is sufficiently large, this is enough to guarantee a copy of $K_h$ induced on vertex set $\{b_1,b_2,\cdots,b_h\}\subseteq N_{B'}(z)$. Since each $d(b_i,b_j)\geq h$, we can now easily find a copy of $H$ while avoiding degeneracy.
\end{proof}

\section{Subdivisions of bipartite graphs}

As a warm-up to the main result, we now prove Theorem~\ref{thm:1subdiv} for subdivisions of bipartite graphs. This is considerably simpler than the general case and we obtain a significantly better bound. To optimise the argument, we will need the following auxiliary lemma. 

\begin{lemma} \label{lem:inj}
Suppose that $G$ is a graph with $n$ vertices and $m \geq n^{2 - 1/t + \delta}$ edges for some positive constant $\delta$. If $G$ contains $\lambda n^{s+t}$ labelled copies of $K_{s,t}$, where $s \leq t$, then it contains at most $(1 + o(1)) \lambda^{e(L)/st} n^{|L|}$ labelled copies of any subgraph $L$ of $K_{s,t}$.
\end{lemma}

\begin{proof}
Let $\mathcal{S}$ be the set of ordered $s$-tuples of distinct vertices in $G$. Then, for $n$ sufficiently large, the number of labelled copies of $K_{s,t}$ in $G$ is
\begin{align*}
\sum_{S \in \mathcal{S}} t! \binom{|N(S)|}{t} & \geq s! t! \binom{n}{s} \binom{\mathbb{E}_S |N(S)|}{t} = s! t! \binom{n}{s} \binom{\sum_v \binom{\deg(v)}{s}/\binom{n}{s}}{t}\\
& \geq s! t! \binom{n}{s} \binom{n \binom{\mathbb{E}_v \deg(v)}{s}/\binom{n}{s}}{t} \geq s! t! \binom{n}{s} \binom{n^{\delta s}}{t},
\end{align*}
where we used convexity twice and in the last inequality that $\mathbb{E}_v \deg(v) \geq 2n^{1-1/s + \delta}.$

Let $\ell = \log n$ and let $\mathcal{S}'$ be the subset of $\mathcal{S}$ consisting of those $S$ for which $|N(S)| \leq \ell$. Then the number of labelled copies of $K_{s,t-1}$ in $G$ satisfies
\begin{align*}
\sum_{S \in \mathcal{S}} (t-1)! \binom{|N(S)|}{t-1} & = \sum_{S \in \mathcal{S} \setminus \mathcal{S}'} (t-1)! \binom{|N(S)|}{t-1} + \sum_{S \in \mathcal{S}'} (t-1)! \binom{|N(S)|}{t-1} \\
& \leq \frac{1}{\ell - t} \sum_{S \in \mathcal{S}} t! \binom{|N(S)|}{t} + s! t! \binom{n}{s} \binom{\ell}{t}\\
& = o\left(\sum_{S \in \mathcal{S}} t! \binom{|N(S)|}{t} \right).
\end{align*}
That is, the number of labelled copies of $K_{s,t-1}$ in $G$ is of asymptotically lower order than the number of labelled copies of $K_{s,t}$ in $G$. A similar argument also works for estimating the number of labelled copies of $K_{s-1,t}$ in $G$ by using the assumption $m \geq n^{2 - 1/t + \delta}$. Using these observations repeatedly, we may then conclude that the number of labelled copies of $K_{s,t}$ in $G$ is equal to $(1 + o(1))|\Hom(K_{s,t}, G)|$.

To finish the proof, we use the fact that $K_{s,t}$ is a weakly norming graph (see~\cite{CL16, H10}), which implies that
\[\left(\frac{|\Hom(L, G)|}{|G|^{|L|}}\right)^{1/e(L)} \leq \left(\frac{|\Hom(K_{s,t}, G)|}{|G|^{s+t}}\right)^{1/st}\]
for any subgraph $L$ of $K_{s,t}$. Substituting in $|\Hom(K_{s,t}, G)| = (1 + o(1)) \lambda n^{s+t}$ then shows that
\[|\Hom(L, G)| \leq (1+ o(1)) \lambda^{e(L)/st} n^{|L|},\]
which yields the required conclusion.
\end{proof}

We now prove the main result of this section.

\begin{theorem}\label{thm:bipartite}
For any $2\leq s\leq t$, there exists a constant $C$ such that if $H_{s,t}$ is the subdivision of the complete bipartite graph $K_{s,t}$, then $\ex(n,H_{s,t})\leq Cn^{3/2-1/12t}$.
\end{theorem}

\begin{proof}
Let $G$ be an $n$-vertex graph with $Cn^{3/2-c}$ edges, where $c=1/(4s + 8t) \geq 1/12t$ and $C$ is taken sufficiently large. By taking $C n_0^{3/2 - c} \geq n_0^2$, we may assume that $n \geq n_0$ for $n_0$ sufficiently large and subsume any loss into the constant $C$. By Lemma~\ref{lem:regular}, we may also assume that $G$ is a $K$-almost-regular bipartite graph on $A\cup B$ with $\frac{1}{2}n\leq |A|\leq\frac{2}{3}n$ and $e(G)=C n^{3/2-c}$ for absolute constants $C$ and $K$. Note that the weighted neighbourhood graph $G_A$ on $A$ has total weight $d|A|^2$ with $d \geq \alpha n^{-2c}$ for some $\alpha>0$ depending on $C$.  

If $\sum_{uv\in\binom{A}{2}} d(u,v)^2\geq n^{2-2c+\varepsilon}$, then Lemma~\ref{lem:manyC4} implies that $G$ contains a copy of $H_{s, t}$ for $n$ sufficiently large. Hence, we may assume that $\sum_{uv\in\binom{A}{2}} d(u,v)^2\leq n^{2-2c+\varepsilon}$. If we define
\begin{align*}
F_\varepsilon:=\left\{uv\in \binom{A}{2}:d(u,v)\geq n^{2\varepsilon}\right\}
\end{align*}
to be the set of `heavily-weighted' edges, then Lemma~\ref{lem:heavy} gives that
\begin{align*}
\sum_{uv\in F_\varepsilon}d(u,v)\leq n^{2-2c-\varepsilon} \leq \beta d |A|^2 n^{-\varepsilon}
\end{align*}
for a positive constant $\beta$. Therefore, for $n$ sufficiently large,
\begin{align*}
\sum_{uv\notin F_\varepsilon} d(u,v) \geq (1-\beta n^{-\varepsilon}) d |A|^{2}\geq d |A|^2/2.
\end{align*}
In other words, by deleting the `heavily-weighted' edges in $F_\varepsilon$, we do not lose much weight. Let $G_A'$ be the (simple) graph obtained by `simplifying' the weighted graph $G_A\setminus F_\varepsilon$, taking $uv\in E(G_A')$ if and only if $0<d(u,v)<n^{2\varepsilon}$. Then $G_A'$ has at least $n^{-2 \varepsilon} d |A|^2/2 = \Omega(n^{2-2c-2\varepsilon})$ edges. 

By the K\H{o}v\'{a}ri--S\'{o}s--Tur\'{a}n theorem~\cite{KST54} and a supersaturation result of Erd\H{o}s and Simonovits~\cite{ESi83}, there exist $\lambda n^{s + t} = \Omega(q^{st}n^{s+t})$ labelled copies of $K_{s,t}$ in $G'_A$, where $q=2e(G_A')/|A|^2=\Omega(n^{-2c-2\varepsilon})$, provided $2(c+\varepsilon)<1/s$. Every such copy of $K_{s,t}$ extends to a homomorphism $\phi:V(H_{s,t})\rightarrow V(G)$ injective on the set of non-subdividing vertices, but possibly degenerate on the degree 2 vertices. 

Let $\Phi$ be the set of such degenerate homomorphisms extended from a copy of $K_{s,t}$ in $G_A'$. By definition, for each $\phi\in \Phi$, there exist two distinct edges $uv$ and $u'v'$ in $K_{s,t}$ such that the subdividing vertices $x_{uv}$ and $x_{u'v'}$ are mapped to the same vertex $b\in B$. Assuming, without loss of generality, that $u$, $v$ and $v'$ are all distinct, we have that $\phi(u),\phi(v)$ and $\phi(v')$ are all in the neighbourhood of $b$. Note now that there are at most $|B| (3K C n^{1/2 - c})^3=O(n^{5/2-3c})$ ways of choosing a $K_{1,3}$ with the central vertex in $B$ and each copy of $K_{1,3}$ extends in at most $M n^{(st-2)2\varepsilon}$ ways to a homomorphic copy of $H_{s,t}$, where $M$ is the number of labelled subgraphs of $G'_A$ equal to either $K_{s-2, t-1}$ or $K_{s-1, t-2}$ and we use the fact that there are at most $n^{2\varepsilon}$ choices for each subdividing vertex (because the copy of $K_{s,t}$ was in $G_A'$). By Lemma~\ref{lem:inj}, $M = O(\lambda^{(s-2)(t-1)/st} n^{s+t-3})$. Therefore, by the choice $c=1/(4s + 8t)$ and $\varepsilon=1/(7st)^2$, 
\begin{align*}
M n^{5/2 - 3c + 2(st-2)\varepsilon} = O(\lambda^{1 - (s+2t-2)/st} n^{s + t - 1/2 - 3c + 2(st-2)\varepsilon}) = o(\lambda n^{s+t}),
\end{align*}
where we used $\lambda = \Omega(q^{st}) = \Omega(n^{-(2c+2\varepsilon)st})$. 
Hence, there always exists a non-degenerate copy of $H_{s,t}$, as required. 
\end{proof}

The lower bound coming from a simple application of the probabilistic deletion method is $\ext(n, H_{s,t}) = \Omega_{s,t} (n^{3/2 - (s+t-3/2)/(2st - 1)}) \geq \Omega_{s,t}(n^{3/2 - 1/2s - 1/2t})$, so our bound is reasonably close to best possible. Indeed, when $s = t$, we have
\[c_t n^{3/2 - 1/t} \leq \ext(n, H_{t,t}) \leq C_t n^{3/2 - 1/12t}\]
for some positive constants $c_t$ and $C_t$. However, the true value of $\ext(n, H_{t,t})$ seems likely to lie somewhere strictly between these two extremes.

\section{The subdivision of $K_t$} 

Since every $t$-vertex $C_4$-free bipartite graph $H$ with maximum degree two on one side is a subgraph of the subdivision of $K_t$, the following theorem implies Theorem~\ref{thm:1subdiv}. 

\begin{theorem}\label{thm:K_t}
For every integer $t \geq 3$, there exists a constant $C$ such that if $H_t$ is the subdivision of $K_t$, then $\ext(n, H_t) \leq Cn^{3/2-1/6^t}$.
\end{theorem}

We first consider the case $t=3$. This already contains all of the essential ideas needed to generalise to higher $t$, but avoids some cumbersome notation. We should note that our estimate is far from the best known (and in this case optimal) result $\ex(n,C_6)=O(n^{4/3})$ given by Bondy and Simonovits~\cite{BS74}. With that said, we have made no great attempt to optimise, because the important point for us is that our method generalises to prove results not accessible to the Bondy--Simonovits method and its generalisations~\cite{FS83}.

\begin{theorem}\label{thm:K_3}
There exists a constant $C$ such that $\ext(n, C_6) \leq Cn^{1.45}$.
\end{theorem}

\begin{proof}
Let $G$ be an $n$-vertex graph with $Cn^{3/2-c}$ edges, where $C$ will be chosen sufficiently large and $c=1/20$. We may assume that $n \geq n_0$ for $n_0$ sufficiently large by subsuming any loss into the constant $C$.
By Lemma~\ref{lem:regular}, we may also assume that $G$ is a $K$-almost-regular balanced bipartite graph on $A\cup B$ with $\frac{1}{2} n \leq |A| \leq \frac{2}{3}n$ and $e(G) = p|A||B|$, where $p = Cn^{-1/2-c}$ and $C$ and $K$ are absolute constants. Lemma~\ref{lem:local_dense} implies that the weighted neighbourhood graph $W$ on $A$ is $(\rho,d)$-dense, where $d=\alpha n^{-2c}$ for some constant $\alpha$ and $\rho \leq n^{-1/4}$ for $n$ sufficiently large.

If $\sum_{uv\in\binom{A}{2}} d(u,v)^2\geq n^{2-2c+\varepsilon/2}$, then Lemma~\ref{lem:manyC4} implies that $G$ contains a copy of $C_6$ for $n$ sufficiently large. Thus, we may assume that $\sum_{uv\in\binom{A}{2}} d(u,v)^2< n^{2-2c+\varepsilon/2}$. If, for $\varepsilon>0$ to be chosen later, we now define
 \begin{align*}
 F_\varepsilon:=\left\{uv\in \binom{A}{2}:d(u,v)\geq n^{\varepsilon}\right\},
 \end{align*}
then Lemma~\ref{lem:heavy} implies that $\sum_{uv\in F_\varepsilon} d(u,v) \leq \beta d n^{2-\varepsilon/2}$ for a positive constant $\beta$.

Let $W_\varepsilon$ be the weighted graph on $A$ obtained from $W$ by deleting all weighted edges $uv\in F_\varepsilon$. Then $\sum_{uv\in \binom{A}{2}}W_\varepsilon(u,v)\geq d |A|^2/4 - \beta dn^{2 - \varepsilon/2} \geq d|A|^2/8$, where we used that $\sum_{uv\in \binom{A}{2}}W(u,v)\geq d \binom{|A|}{2} \geq d|A|^2/4$.
Therefore, by Lemma~\ref{lem:large_nbd}, there is a set $U$ with $|U| \geq dn^{1-\varepsilon}/16$ such that if $A_u =\{v\in A: W_{\varepsilon}(u,v)>0\}$, then $|A_u|\geq dn^{1-\varepsilon}/16$ for all $u \in U$.

Our intention is to apply Lemma~\ref{lem:manyC4asym_2} with 
$A_\text{\ref{lem:manyC4asym_2}}=A_u$, where the set $B$ remains the same.
By Lemma~\ref{lem:ordered_pairs},
\begin{align*}
\sum_{vw\in\binom{A_u}{2}}d(v,w)
\geq \frac{1}{4}|\Hom^*(K_{2,1},G[A_u,B])|,
\end{align*}
provided $n$ is sufficiently large. Now suppose that $|\Hom^*(K_{2,1},G[A_u,B])|\geq |\Hom^*(K_{1,2},G[A_u,B])|$.
If $\delta>0$ satisfies 
\begin{align*}
	|\Hom(C_4,G[A_u,B])| \geq  2n^{\delta}|\Hom^*(K_{2,1},G[A_u,B])|\geq n^{\delta}|\Hom(K_{1,2},G[A_u,B])|,
\end{align*}
then there exists $C_6$ in $G[A_u,B]$ by Lemma~\ref{lem:manyC4asym_2}.
Thus, for any positive $\delta$, it follows that
\[
\sum_{vw\in\binom{A_u}{2}}d(v,w)^2\leq n^{\delta}|\Hom^*(K_{2,1},G[A_u,B])|\leq 4n^{\delta}\sum_{vw\in\binom{A_u}{2}}d(v,w).
\] 
Let $F_{\delta,u}$ be the set $F_{\delta,u}:=\left\{vw\in \binom{A_u}{2}:d(v,w)\geq n^{2\delta}\right\}$.
Then Lemma~\ref{lem:heavy} implies that
\begin{align}\label{eq:small_F}
	\sum_{vw\in F_{\delta,u}}d(v,w)\leq 4n^{-\delta}\sum_{vw\in\binom{A_u}{2}}d(v,w).
\end{align}

Suppose now that $|\Hom^*(K_{2,1},G[A_u,B])< |\Hom^*(K_{1,2},G[A_u,B])|$.
Again, if there is a positive $\delta$ such that
\[
|\Hom(C_4,G[A_u,B])|\geq  2n^{\delta}|\Hom^*(K_{1,2},G[A_u,B])|\geq n^{\delta}|\Hom(K_{1,2},G[A_u,B])|,
\]
then Lemma~\ref{lem:manyC4asym_2} gives a copy of $C_6$ in $G[A_u,B]$.
Hence, since $d(v) \leq Kpn$ for all $v \in A_u$,
\begin{align}\label{eq:C4<K12}
\sum_{vw\in\binom{A_u}{2}}d(u,v)^2\leq n^{\delta}|\Hom^*(K_{1,2},G[A_u,B])|\leq K^2|A_u|n^{2+\delta}p^2.
\end{align}
Combining the $(\rho, d)$-denseness of the neighbourhood graph $W$ with~\eqref{eq:C4<K12}, we obtain
\[
\sum_{vw\in\binom{A_u}{2}}d(v,w)^2\leq \frac{K^2|A_u|n^{2+\delta}p^2}{d\binom{|A_u|}{2}}
\sum_{vw\in\binom{A_u}{2}}d(v,w)=O(n^{2c+\varepsilon+\delta})\sum_{vw\in\binom{A_u}{2}}d(v,w).
\]
Let $\xi = c + \varepsilon + \delta$. Then, again by Lemma~\ref{lem:heavy}, 
\begin{align}\label{eq:small_F2}
	\sum_{vw\in F_{\xi,u}}d(v,w)\leq O(n^{-\varepsilon - \delta})\sum_{vw\in\binom{A_u}{2}}d(v,w).
\end{align}
Therefore, in both cases \eqref{eq:small_F} and \eqref{eq:small_F2}, by the $(\rho, d)$-denseness of the neighbourhood graph $W$,
\begin{align*}
\sum_{vw\notin F_{\xi,u}}d(v,w) = \sum_{vw \in \binom{A_u}{2}}d(v,w) - \sum_{vw\in F_{\xi,u}}d(v,w) 
%\geq d|A_u|^2/4 - O(d|A_u|^2 n^{-\xi}) 
\geq d|A_u|^2/8
\end{align*} 
for $n$ sufficiently large.

Let us count the number of labelled triangles $(u,v,w)$ in the neighbourhood graph 
such that $u\in U$, $uv, uw \notin F_{\varepsilon}$, and $vw\notin F_{\xi,u}$.
For each $u\in U$, there exist at least $|A_u|^2$ pairs $(v,w)$ such that $uv,uw\notin F_{\varepsilon}$
and, among these, at least $d|A_u|^2n^{-2\xi}/8$ edges are in $\binom{A_u}{2}\setminus F_{\xi,u}$, as the sum of weights $\sum_{vw \in \binom{A_u}{2}\setminus F_{\xi,u}}d(v,w)$ 
is at least $d|A_u|^2/8$ and each edge weight does not exceed $n^{2\xi}$.
Therefore, in total, there are at least
\begin{align*}
\sum_{u\in U}\frac{d|A_u|^2n^{-2\xi}}{8} \geq \frac{d|U|}{8}\left(\frac{dn^{1-\varepsilon}}{16}\right)^2 n^{-2\xi}
\geq  
2^{-15} d^4n^{3-3\varepsilon-2\xi}
=\frac{\alpha^4}{2^{15}}n^{3-10c-5\varepsilon-2\delta}
\end{align*}
labelled triangles, each of which extends to a homomorphism in $\Hom(C_6,G)$ that is injective on $A$. Such a homomorphism is degenerate if and only if the joint neighbours of at least two of $uv$, $vw$, $wu$ meet. If that happens, the homomorphism is fixed by a choice of vertex in $B$ (at most $n$ choices), a choice of three neighbours of this vertex (at most $(Kpn)^3$ choices) and a choice for the common neighbour of the third pair (at most $3 n^{2\xi}$ choices). Overall, this gives at most $3K^3 p^3 n^{4 + 2 \xi} = O(n^{5/2 - c + 2\varepsilon + 2\delta})$ degenerate homomorphisms. Therefore, if $5/2 - c + 2\varepsilon + 2\delta < 3-10c-5\varepsilon-2\delta$,
there exists a non-degenerate copy of $C_6$. Recalling that $c=1/20$ and choosing $\varepsilon = \delta =1/300$ suffices for this purpose.
\end{proof}

In order to generalise this result to $K_t$, we need some further notation. We say that a subset $A'$ of $A$ is \emph{$L$-bounded} if
\[
|\Hom(C_4,G[A',B])|\leq L|\Hom^*(K_{2,1},G[A',B])|.
\]
The following lemma distills out one of the key steps in the proof above.

\begin{lemma}\label{lem:dichotomy}
For every $\varepsilon>0$, there exists $n_0$ such that if $n \geq n_0$, $G$ is an $n$-vertex $K$-almost-regular balanced bipartite graph on $A\cup B$ which is $H_{t}$-free and $A'$ is a subset of $A$ with density $\gamma:=|A'|/|A|$,
then $A'$ is $n^{\varepsilon}/\gamma$-bounded. 
\end{lemma}
\begin{proof}
Let $G':=G[A',B]$ and let $p:=e(G)/|A||B|$.
If $|\Hom^*(K_{1,2},G')|\leq |\Hom^*(K_{2,1},G')|$, then Lemma~\ref{lem:manyC4asym_2}
immediately implies that $A'$ is $n^{\varepsilon/2}$-bounded.
Suppose now that $|\Hom^*(K_{1,2},G')|> |\Hom^*(K_{2,1},G')|$.
Then, again by Lemma~\ref{lem:manyC4asym_2},
\begin{align}\label{eq:dichotomy}
|\Hom(C_4,G')|\leq n^{\varepsilon/2}|\Hom^*(K_{1,2},G')|\leq K^2p^2n^{2+\varepsilon/2}|A'|.
\end{align}
On the other hand,
\[
|\Hom^*(K_{2,1},G')|=\sum_{b\in B}\deg_{G'}(b)^2 \geq \frac{1}{|B|}e(A',B)^2\geq \frac{|A'|^2|B|p^2}{K^2},
\]
where the first inequality is by convexity and the second follows from $K$-almost regularity of $G$.
Combining this with \eqref{eq:dichotomy} implies that
\[
|\Hom(C_4,G')|\leq K^2p^2n^{2+\varepsilon/2}|A'|\leq n^{\varepsilon}\gamma^{-1}|\Hom^*(K_{2,1},G')|
\]
for $n$ sufficiently large, as required.
\end{proof}

For $\xi>0$ and $u\in A$, define $\eta_\xi(u):=\{a\in A: 1\leq d(u,a)<n^{\xi}\}$. Let $\xi_1<\xi_2<\cdots<\xi_{j}$ be an increasing sequence of $j$ positive numbers and $\delta>0$. A $j$-tuple $J=(a_1,\cdots,a_j)\in A^j$ is said to be $(\xi_1,\xi_2,\cdots,\xi_{j};\delta)$-\emph{good} if, for every $1\leq i < i'\leq j$, $a_{i'} \in \eta_{\xi_i}(a_i)$ and there exist at least $|A|^{1-\delta}$ vertices $a$ such that $a\in \eta_{\xi_i}(a_i)$ for all $i=1,2,\cdots,j$. We say that each $a\in \cap_{i=1}^j\eta_{\xi_i}(a_i)$ \emph{extends} the $(\xi_1,\xi_2,\cdots,\xi_{j};\delta)$-good $j$-tuple $J$. That is, $\{a_1,\cdots,a_j\}$ induce a copy of $K_j$ in the neighbourhood graph on $A$, where the codegrees $d(a_i,a_{i'})$ are not too large and there exist many choices for proceeding further. The core of the proof of Theorem~\ref{thm:K_t} is now contained in the following lemma, proving the existence of many good $t$-tuples in sufficiently dense $H_{t+1}$-free graphs, where again $H_{t+1}$ is the subdivision of $K_{t+1}$.

\begin{lemma}\label{lem:iterate}
For every integer $t \geq 1$ and any positive constants $K$ and $\delta$ with $\delta<1/4$, there exist positive constants $C$, $c$ and $\xi_1,\cdots,\xi_t$ such that if $G$ is an $n$-vertex $K$-almost-regular balanced bipartite graph on $A\cup B$ with $e(G)=Cn^{3/2-c}$ edges which is $H_{t+1}$-free, then the number of $(\xi_1,\xi_2,\cdots,\xi_{t};\delta)$-good $t$-tuples is at least $ |A|^{t-\delta}$.
\end{lemma}

\begin{proof}
As usual, we may assume that $n \geq n_0$ for $n_0$ sufficiently large by subsuming any loss into the constant $C$. We will prove the result by induction on $j$, showing that for every integer $1 \leq j \leq t$ and any positive constants $K$ and $\delta_j$ with $\delta_j < 1/4$, there are $C$, $c$ and $\xi_1,\cdots,\xi_j$ such that if $G$ is an $n$-vertex $K$-almost-regular balanced bipartite graph on $A\cup B$ with $e(G)= p|A||B|$ edges for $p = Cn^{-1/2-c}$ which is $H_{t+1}$-free, then the number of $(\xi_1,\xi_2,\cdots,\xi_{j};\delta_j)$-good $j$-tuples is at least $|A|^{j-\delta_j}$.

Before starting the induction, it will be useful to note, by applying Lemma~\ref{lem:local_dense} with $\delta = pn/4K$, that if $A' \subseteq A$ has order at least $n^{3/4}\geq 8K/p\geq 2n/\delta$, then 
$$\sum_{uv \in \binom{A'}{2}} d(u,v) \geq \frac{\delta^2}{2n} \binom{|A'|}{2} \geq \frac{C^2}{128K^2} n^{-2c} |A'|^2.$$
For brevity in what follows, we will write $d_0 = C^2 n^{-2c}/128K^2$.

Suppose first that $j=1$. 
By Lemma~\ref{lem:manyC4} with $\varepsilon=\xi_1/2$, 
we have $\sum_{uv\in\binom{A}{2}}d(u,v)^2\leq n^{2-2c+\xi_1/2}$, since otherwise $G$ would contain $H_{t+1}$ for $n$ sufficiently large. We may therefore apply Lemma~\ref{lem:heavy} with $M_{\ref{lem:heavy}} = n^{\xi_1}$ and $S_{\ref{lem:heavy}} = n^{2-2c + \xi_1/2}$ to conclude that
\begin{align*}
\sum_{uv\in F_M} d(u,v) = O(n^{-\xi_1/2}d_0|A|^2).
\end{align*}
Since $\sum_{uv\in\binom{A}{2}}d(u,v) \geq d_0 |A|^2$, we see that $\sum_{uv\notin F_M}d(u,v) \geq d_0 |A|^2/2$  for $n$ sufficiently large, so we may apply Lemma~\ref{lem:large_nbd} to conclude that there is $U\subseteq A$ with $|U|\geq d_0|A|n^{-\xi_1}/2$ such that
\begin{align*}
|\eta_{\xi_1}(u)|=|\{v\in A: 1\leq d(u,v)<n^{\xi_1}\}|\geq d_0|A|n^{-\xi_1}/2.
\end{align*}
Thus, taking $2c+\xi_1<\delta_1$, we have that $|U|\geq |A|^{1-\delta_1}$ and $|\eta_{\xi_1}(u)|\geq |A|^{1-\delta_1}$ for each $u\in U$ and $n$ sufficiently large.

Suppose now that $\xi_1,\cdots,\xi_{j}$ are such that at least $|A|^{j-\delta_j}$ $(\xi_1,\cdots,\xi_{j};\delta_j)$-good $j$-tuples exist, for some $\delta_j$ to be specified later. For brevity, we will say that a $j$-tuple $J$ is \emph{good} if it is $(\xi_1,\cdots,\xi_{j};\delta_j)$-good and let $\J$ be the set of all good $j$-tuples. For each good $J$, denote by $\K_{j+1}(J)$ the set of all $a\in A$ that extend $J$. Then $|\K_{j+1}(J)|\geq |A|^{1-\delta_j}$ by definition.

By applying Lemma~\ref{lem:dichotomy} with $\varepsilon = \delta_j/2$, 
we see that $\K_{j+1}(J)$ is $n^{3\delta_j/2}$-bounded for $n$ sufficiently large.
Define 
$$F_{J,\delta_j} := \left\{uv\in \binom{\K_{j+1}(J)}{2}:d(u,v)\geq n^{2\delta_j}\right\}.$$ 
Since $\sum_{uv\in \binom{K_{j+1}(J)}{2}} d(u,v)^2 \leq |\Hom(C_4,G[\K_{j+1}(J),B])|$, Lemma~\ref{lem:heavy} and the $n^{3\delta_j/2}$-boundedness of $\K_{j+1}(J)$ imply that
\begin{align}\label{eq:badweight}
\sum_{uv\in F_{J,\delta_j}} d(u,v) \leq n^{-\delta_j/2}
|\Hom^*(K_{2,1},G[\K_{j+1}(J),B])|.
\end{align}
Since $\delta_j < 1/4$, we have $|K_{j+1}(J)| \geq n^{3/4} \geq 2n/\delta$, so
Lemma~\ref{lem:ordered_pairs} implies that 
\[\sum_{uv\in \binom{\K_{j+1}(J)}{2}}d(u,v)\geq \frac{1}{4}|\Hom^*(K_{2,1},G[\K_{j+1}(J),B])|.\]
Together with \eqref{eq:badweight}, this yields
\[
\sum_{uv\notin F_{J,\varepsilon_j}} d(u,v) \geq (1-4n^{-\delta_j/2})
\sum_{uv\in \binom{\K_{j+1}(J)}{2}}d(u,v)
\geq \frac{1}{2}\sum_{uv\in \binom{\K_{j+1}(J)}{2}}d(u,v).
\]
Moreover, by Lemma~\ref{lem:local_dense}, 
$\sum_{uv\in \binom{\K_{j+1}(J)}{2}}d(u,v)\geq d_0|\K_{j+1}(J)|^2$ and hence, 
for $n$ sufficiently large, $\sum_{uv \notin F_{J,\varepsilon_j}} d(u,v)\geq d_0|\K_{j+1}(J)|^2/2$.
Therefore, by Lemma~\ref{lem:large_nbd}, there is a subset $U\subseteq |\K_{j+1}(J)|$ with 
$$|U|\geq n^{-2\delta_j} d_0|K_{j+1}(J)|/2 =\Omega(|A|^{1 - 3\delta_j - 2c})$$
such that, for each $u\in U$,
\begin{align*}
|\{v\in\K_{j+1}(J):0<d(u,v)<n^{2\delta_j}\}|\geq n^{-2\delta_j} d_0|\K_{j+1}(J)|/2 = \Omega(|A|^{1 - 3\delta_j - 2c}).
\end{align*}
Taking $\delta_{j+1}>4\delta_j+2c$ and $n$ sufficiently large, we see that $(J,u)$ is $(\xi_1,\cdots,\xi_j,2\delta_j;\delta_{j+1})$-good for each $u\in U$ and the number of $(\xi_1,\cdots,\xi_j,2\delta_j;\delta_{j+1})$-good $(j+1)$-tuples is at least
$$
|A|^{j-\delta_j}|U|= \Omega(|A|^{j+1-4\delta_j-2c}) \geq |A|^{j+1 - \delta_{j+1}}.
$$
Taking $\delta_j = \delta/6^{t-j}$ and $c = \xi_1 = \delta/6^t$, we see that
\[
4 \delta_j + 2c = \frac{4 \delta}{6^{t-j}} + \frac{2 \delta}{6^t} < \frac{\delta}{6^{t-j-1}} = \delta_{j+1},
\]
$2c + \xi_1 < \delta_1$ and $\delta_j < 1/4$ for all $j$, so the necessary conditions hold. 
For future use, we also note that since $\xi_{j} = 2 \delta_{j-1}$, 
we have $3 \xi_{j} = \delta/6^{t-j}$ for all $j \geq 1$.
\end{proof}

We are now ready to complete the proof of Theorem~\ref{thm:K_t}.

\begin{proof}[Proof of Theorem~\ref{thm:K_t}]
Let $G$ be an $n$-vertex graph with $Cn^{3/2-c}$ edges, where $C$ will be chosen sufficiently large and $c = 6^{-t}$. We may assume that $n \geq n_0$ for $n_0$ sufficiently large by subsuming any loss into the constant $C$.
By Lemma~\ref{lem:regular}, we may also assume that $G$ is a $K$-almost-regular balanced bipartite graph on $A\cup B$ with $e(G)= p|A||B|$, where $p = Cn^{-1/2-c}$ and $C$ and $K$ are absolute constants.

Suppose that $G$ is $H_{t}$-free. Then Lemma~\ref{lem:iterate} implies that there are at least $|A|^{t-1-\delta}$ $(t-1)$-tuples that are
$(\xi_1,\xi_2,\cdots,\xi_{t-1};\delta)$-good, where $\xi_j = 6^{j} c/3$ for $1\leq j\leq t-1$ and $\delta = 6^{t-1} c = 1/6$. We say that a $(t-1)$-tuple $T$ is good if $T$ is $(\xi_1,\xi_2,\cdots,\xi_{t-1};\delta)$-good. By definition, a good $(t-1)$-tuple $T=(a_1,a_2,\cdots,a_{t-1})$ together with any $u\in A$ that extends $T$ induces a copy of $K_t$ in the neighbourhood graph which can be extended to a homomorphic copy of $H_t$ in $G$.
Let $\Psi$ be the set of all homomorphisms from $H_t$ to $G$ constructed from a good $T$ and a vertex that extends $T$. There are at least $|A|^{t-2\delta}$ homomorphic copies of $H_t$ in $\Psi$. Any degenerate copy must contain a copy of $K_{1,3}$, which extends to at most $n^{t-3+\kappa}$ homomorphic copies of $K_t$, where 
\[
\kappa=\sum_{j=1}^{t-1}(t-j)\xi_j = \frac{2c}{25}(6^t-5t-1)< 6^{t-1} c.
\]
Therefore, there are at most
\begin{align*}
K^3n^4p^3 n^{t-3+\kappa}=O(n^{t-1/2-3c+\kappa})
\end{align*}
degenerate homomorphisms in $\Psi$.
As $t-2\delta>t-1/2-3c+\kappa$ for our choice of $\xi_1,\xi_2,\cdots,\xi_{t-1}$, $c$ and $\delta$, there exists a non-degenerate copy of $H_{t}$ in $\Psi$.
\end{proof}

\section{Concluding remarks}

Our investigations raise many open problems. The first, most obvious, question is to give a better estimate for $\ext(n, H_t)$, where $H_t$ is the subdivision of $K_t$. We have
\[c_t n^{3/2 - (t-3/2)/(t^2 - t - 1)} \leq \ext(n, H_t) \leq C_t n^{3/2 - 1/6^t},\]
where the upper bound is Theorem~\ref{thm:K_t} and the lower bound follows from a simple application of the probabilistic deletion method (see, for instance,~\cite[Section 2.5]{FS13}). Ideally, we would like to bring the two bounds in line with each other, but a first step might be to improve the upper bound to $\ext(n, H_t) \leq C_t n^{3/2 - \delta_t}$ where $\delta_t^{-1}$ is bounded by a polynomial in $t$.

While interesting in its own right, the lower bound is also interesting because of its connection with the study of pseudorandom graphs. A surprising construction of Alon~\cite{A94} shows that there are highly pseudorandom triangle-free graphs with $n$ vertices and density $\Omega(n^{-1/3})$, surprising because this density is much higher than the $\Omega(n^{-1/2})$ produced by random means. Quite recently, the first author~\cite{C17} found an alternative construction of such graphs by `unsubdividing' a standard construction of $C_6$-free graphs. A similar approach might allow us to construct pseudorandom $K_t$-free graphs of high density if only we had better constructions for $H_t$-free graphs than those given by the probabilistic method. A first aim would be to beat the construction of Alon and Krivelevich~\cite{AK97}, which gives optimally pseudorandom $K_t$-free graphs with $n$ vertices and density $\Omega(n^{-1/(t-2)})$, but it is plausible that such graphs exist all the way up to density $\Omega(n^{-1/(2t-3)})$.

For the more general Conjecture~\ref{conj:main}, it seems that new ideas will be needed. Our proof relies in an essential way on the fact that the neighbourhood graph of a bipartite graph with high minimum degree is $(\rho, d)$-dense for appropriate parameters $\rho$ and $d$ and then that this condition is sufficient for finding copies of any fixed graph $H$. If we move to $r \geq 3$, the neighbourhood graph must be replaced with an $r$-uniform hypergraph and the $(\rho, d)$-denseness condition, or rather its obvious hypergraph analogue, is then completely insufficient for finding the necessary subhypergraphs (see~\cite{KNRS10} for a discussion of this point). However, there are certain special cases which might still be amenable to our methods. To say more, we state a slightly weaker version of our main conjecture.

\begin{conjecture} \label{conj:weak}
For any bipartite graph $H$ between vertex sets $A$ and $B$ such that every vertex in $B$ has maximum degree $r$ and there is no $K_{r,2}$ with the side of order $2$ in $B$, there exist positive constants $C$ and $\delta$ such that 
\[\ext(n,H) \leq C n^{2 - 1/r - \delta}.\]
\end{conjecture}

Now, given a hypergraph $H$, we define its {\it subdivision} to be the bipartite graph between $V(H)$, the vertices of $H$, and $E(H)$, the edges of $H$, where we join $v \in V(H)$ and $e \in E(H)$ if and only if $v \in e$. Conjecture~\ref{conj:weak} then says that if $H$ is an $r$-uniform hypergraph, its subdivision has extremal number at most $C n^{2 - 1/r - \delta}$ for some positive constants $C$ and $\delta$. Stated in this form, it clearly suffices to prove the conjecture for subdivisions of the complete $r$-uniform hypergraphs $K_t^{(r)}$, but it also allows one to consider other cases, such as when $H$ is a linear hypergraph, which might be more accessible. Indeed, these linear hypergraphs, where any pair of edges intersect in at most one vertex, are controllable using a local density condition~\cite{KNRS10}, so our techniques might conceivably apply. Our techniques certainly do apply for subdivisions of $r$-partite $r$-uniform hypergraphs, where the proof of Theorem~\ref{thm:bipartite} easily extends to prove the following result. If we are not concerned with optimising the result, we may even do without the hypergraph analogue of Lemma~\ref{lem:inj}.

\begin{theorem}\label{thm:rpartite}
For any $r, t \geq 2$, there exist positive constants $C$ and $\delta$ such that if $H_{t,t, \dots, t}$ is the subdivision of the $r$-partite $r$-uniform hypergraph $K_{t,t, \dots, t}$, then $\ex(n,H_{t,t, \dots, t})\leq Cn^{2 - 1/r - \delta}$.
\end{theorem}

Beyond the $r = 2$ case, there is one other family for which we can easily verify Conjecture~\ref{conj:weak}, namely, when $H$ is the subdivision of $K_{r+1}^{(r)}$, which is equal to $K_{r+1, r+1}$ with a perfect matching removed. The $r = 3$ case corresponds to the cube $Q_3$, for which it is known~\cite{ES70} that $\ext(n, Q_3) \leq C n^{8/5}$. More generally, if we write $S_r$ for the subdivision of $K_{r+1}^{(r)}$, a result of Erd\H{o}s and Simonovits~\cite[Theorem 2]{ES70} allows one to derive a bound for the extremal number of $S_r$ from a bound for the extremal number of $S_{r-1}$. This recursion shows that $\ext(n, S_r) \leq C_r n^{2 - 2/(2r-1)}$, which is of the required form.

Moving onto longer subdivisions, given a multigraph $H$ (without loops), we define the {\it $k$-subdivision} $H^k$ of $H$ to be the graph obtained from $H$ by replacing the edges of $H$ with internally disjoint paths of length $k+1$. We make the following conjecture, which would improve and generalise the result of Jiang and Seiver~\cite{JS12} discussed in the introduction that $\ext(n, K_t^{k}) = O_{k,t}(n^{1 + 16/(k+1)})$ for any odd $k$.

\begin{conjecture}
For any multigraph $H$ and any odd $k \geq 1$, there exists a constant $C$ such that 
\[\ext(n, H^{k}) \leq C n^{1 + 1/(k+1)}.\]
\end{conjecture}

Perhaps the simplest case of this conjecture is when $H$ is a pair of vertices connected by two edges, where $H^{k}$ corresponds to a cycle of length $2k+2$. Since we know that $\ext(n, C_{2k+2}) = O_k(n^{1 + 1/(k+1)})$, the conjecture holds in this case. Similarly, when $H$ is a pair of vertices connected by more than two edges, the conjecture holds by a result of Faudree and Simonovits~\cite{FS83} (and, for sufficiently many edges, is tight by a result of the first author~\cite{C18}). That it should also hold for more general multigraphs seems highly plausible. However, in analogy with Conjecture~\ref{conj:main}, we might expect more.

\begin{conjecture}
For any (simple) graph $H$ and any odd $k \geq 1$, there exist positive constants $C$ and $\delta$ such that
\[\ext(n, H^{k}) \leq C n^{1 + 1/(k+1) - \delta}.\]
\end{conjecture}

\vspace{4mm}
\noindent
{\bf Note added.}
Shortly after this paper was submitted, Janzer~\cite{Ja18} settled one of our problems by showing that $\ext(n, H_t) \leq C_t n^{3/2 - 1/(4t-6)}$ for all $t \geq 3$. This is tight up to the constant for $t = 3$ and the connection with pseudorandom graphs discussed in the concluding remarks suggests that it may even be tight for all $t$.

\vspace{4mm}
\noindent
{\bf Acknowledgements.} This paper was partially written while the first author was visiting the California Institute of Technology as a Moore Distinguished Scholar and he is extremely grateful for their kind support.

\bibliographystyle{abbrv}
\bibliography{references}

\end{document}